\documentclass[a4paper, 12pt]{article}
\usepackage{geometry}                
\usepackage{graphicx}
\usepackage{amssymb,amsmath,scalerel}
\usepackage{amsthm}
\usepackage{thmtools, thm-restate}
\usepackage[plainpages = false]{hyperref}
\usepackage[capitalize]{cleveref}
\usepackage[utf8]{inputenc}
\usepackage{enumerate}
\usepackage{xcolor}
\usepackage{tikz-cd}
\usepackage[backend = biber, style = numeric, giveninits]{biblatex}

\DeclareMathOperator*{\Times}{\scalerel*{\times}{\textstyle\sum}}

\DeclareMathOperator{\Aut}{Aut}
\DeclareMathOperator{\Diag}{Diag}
\DeclareMathOperator{\End}{End}

\DeclareMathOperator{\Fix}{Fix}

\DeclareMathOperator{\GL}{GL}

\DeclareMathOperator{\Id}{Id}
\DeclareMathOperator{\im}{Im}

\DeclareMathOperator{\SpecR}{Spec_R}
\DeclareMathOperator{\SpecP}{Spec_{\Pi}}

\newcommand*{\eg}{e.g.\ }

\newcommand*{\grpgen}[1]{\left\langle{#1}\right\rangle}

\newcommand*{\ie}{i.e.\ }

\newcommand*{\inv}[1]{{#1}^{-1}}
\newcommand*{\J}{\mathcal{J}}

\newcommand*{\N}{\mathbb{N}}

\newcommand*{\Rconj}[1]{\sim_{#1}}
\newcommand*{\Rinf}{R_{\infty}}

\newcommand*{\restr}[2]{{#1}\big|_{#2}}
\newcommand*{\size}[1]{\left| #1 \right|}

\newcommand*{\transpose}[1]{{#1}^\mathsf{T}}

\newcommand*{\ZmodZ}[1]{\Z / #1 \Z}
\newcommand*{\ZnZ}{\ZmodZ{n}}
\newcommand*{\ZpZ}{\ZmodZ{p}}
\newcommand*{\Z}{\mathbb{Z}}

\renewcommand{\phi}{\varphi}
\renewcommand{\P}{\mathcal{P}}

\let\originalleft\left
\let\originalright\right
\renewcommand{\left}{\mathopen{}\mathclose\bgroup\originalleft}
\renewcommand{\right}{\aftergroup\egroup\originalright}

\declaretheorem[style=definition, name = Definition, numberwithin=section]{defin}

\declaretheorem[name = Theorem, sibling=defin]{theorem}
\declaretheorem[name = Lemma, sibling=defin]{lemma}
\declaretheorem[name = Proposition, sibling=defin]{prop}
\declaretheorem[name = Corollary, sibling=defin]{cor}
\declaretheorem[style=remark, name = Remark, numbered = no]{remark}

\declaretheorem[name = Theorem, numbered = no]{theorem*}

\numberwithin{equation}{section}

\crefname{prop}{Proposition}{Propositions}
\crefname{cor}{Corollary}{Corollaries}

\title{The Reidemeister spectrum of finite abelian groups}
\author{Pieter Senden\footnote{Researcher funded by FWO-fellowship fundamental research (file number 1112522N)}}
\date{}                                           

\begin{document}
\maketitle

\begin{abstract}
	For a finite abelian group \(A\), the Reidemeister number of an endomorphism \(\phi\) equals the size of \(\Fix(\phi)\), the set of fixed points of \(\phi\). Consequently, the Reidemeister spectrum of \(A\) is a subset of the set of divisors of \(\size{A}\). We fully determine the Reidemeister spectrum of \(A\), that is, which divisors of \(\size{A}\) occur as the Reidemeister number of an automorphism. To do so, we discuss and prove a more general result providing upper and lower bounds on the number of fixed points of automorphisms related to a given automorphism \(\phi\).
\end{abstract}
\let\thefootnote\relax
\let\thefootnote\relax
\section{Introduction}

Given a group \(G\) and an endomorphism \(\phi\), we define the \(\phi\)-twisted conjugacy relation on \(G\) by stating that \(x, y \in G\) are \(\phi\)-conjugate if there exists a \(z \in G\) such that \(x = z y \inv{\phi(z)}\). If \(x\) and \(y\) are \(\phi\)-conjugate, we write this as \(x \Rconj{\phi} y\). The number of \(\phi\)-conjugacy classes is called the Reidemeister number of \(\phi\) and it is denoted by \(R(\phi)\). Furthermore, we define the Reidemeister spectrum of \(G\) as \(\SpecR(G) := \{R(\psi) \mid \psi \in \Aut(G)\}\).

One of the general objectives is to determine the complete Reidemeister spectrum of a group. There are two extreme cases that can occur: (1) \(G\) has the \(\Rinf\)-property, meaning that \(\SpecR(G) = \{\infty\}\), and (2) \(G\) has full Reidemeister spectrum, meaning that \(\SpecR(G) = \N_{0} \cup \{\infty\}\). The first case in particular has been extensively studied. Non-abelian Baumslag-Solitar groups \cite{FelshtynGoncalves06} and their generalisations \cite{TabackWong08}, certain extensions of linear groups by a countable abelian group \cite{MubeenaSankaran14}, Thompson's group \(F\) \cite{BleakFelshtynGoncalves08} all have the \(\Rinf\)-property; the free nilpotent group \(N_{r, c}\) of rank \(r \geq 2\) and class \(c \geq 1\) has been proven to have \(\Rinf\)-property if and only if \(c \geq 2r\), see \eg \cite{Romankov11,DekimpeGoncalves14}. We refer the reader to \cite{FelshtynNasybullov16} for a more exhaustive list of examples.

For the second extreme case, fewer examples of groups have been found. One family of such group are the groups \(N_{r, 2}\) where \(r \geq 4\) \cite{DekimpeTertooyVargas20}. Finally, groups whose Reidemeister spectrum has been fully determined bu have neither the \(\Rinf\)-property nor full Reidemeister spectrum include the semidirect products \(\Z^{n} \rtimes \ZmodZ{2}\), where \(n \geq 2\) and \(\ZmodZ{2}\) acts by inversion \cite{DekimpeKaiserTertooy19}.

For finite groups, however, neither extreme case can occur. As far as the author knows, there is only little literature concerning twisted conjugacy and Reidemeister numbers on finite groups. A.\ Fel'shtyn and R.\ Hill proved that the Reidemeister number of an endomorphism \(\phi\) of a finite group equals the number of (ordinary) conjugacy classes that are fixed by \(\phi\) \cite[Theorem~5]{FelshtynHill94}, and also discussed Reidemeister zeta functions on finite groups. Still, information about Reidemeister numbers of finite groups can aid in determining the Reidemeister spectrum of infinite groups, since finitely generated residually finite groups can be studied by looking at their finite characteristic quotients.

Given a finite group \(G\), it is theoretically possible to compute its Reidemeister spectrum using a computer; for instance, S.\ Tertooy has developed a GAP-package \cite{TertooyGAP} that has these functionalities. However, for an arbitrary finite group, the only feasible way to do so is to use either the definition or the result by A.\ Fel'shtyn and R.\ Hill relating Reidemeister numbers to fixed conjugacy classes. Either method requires a substantial amount of computation time if the order of \(G\) increases, since one has to determine \(\Aut(G)\), possibly the set of all conjugacy classes, and the Reidemeister number \(R(\phi)\) for each \(\phi \in \Aut(G)\). To reduce this time one can try and find explicit expressions for the Reidemeister spectrum of certain (families of) finite groups or even just methods that do not require to fully compute \(\Aut(G)\) in order to determine \(\SpecR(G)\). The former has been done by the author for split metacyclic groups of the form \(C_{n} \rtimes C_{p}\) where \(p\) is a prime number in \cite{Senden21a}.

The aim of this paper is to completely determine the Reidemeister spectrum of finite abelian groups. We would like to mention that Reidemeister spectra of infinite abelian groups, on the other hand, have already been studied, see \cite[\S 3]{Romankov11},\cite{DekimpeGoncalves15,GoldsmithKarimiWhite19}.

This paper is organised as follows. In Section 2, we recall the necessary results regarding Reidemeister numbers and reduce the problem to finite abelian groups of prime power order. In Section 3, we determine the Reidemeister spectrum of finite abelian \(p\)-groups with \(p\) an odd prime. In Section 4, finally, we determine the Reidemeister spectrum of finite abelian \(2\)-groups by solving a more general problem regarding fixed points of automorphisms of finite abelian \(p\)-groups.



Unless otherwise stated, \(p\) denotes a prime number.
\section{Preliminaries}
\begin{prop}	\label{prop:ReidemeisterNumberEqualsNumberOfFixedPoints}
	Let \(A\) be a finite abelian group and \(\phi \in \End(A)\). Then \(R(\phi) = \size{\Fix(\phi)}\).
\end{prop}
\begin{proof}
	Note that, for all \(x, y \in A\), we have
	\[
		x \Rconj{\phi} y \iff \exists z \in A: x = z + y - \phi(z) \iff x - y \in \im(\Id - \phi).
	\]
	Therefore, \(R(\phi) = [A : \im(\Id - \phi)]\). Since \(A\) is finite, we moreover have that
	\[
		[A : \im(\Id - \phi)] = \frac{\size{A}}{\size{\im(\Id - \phi)}} = \size{\ker(\Id - \phi)} = \size{\Fix(\phi)},
	\]
	by the first isomorphism theorem for groups.
\end{proof}
\begin{cor}	\label{cor:SpecRFiniteAbelianGroupsConsistsOfDivisors}
Let \(A\) be a finite abelian group. Then \(\SpecR(A) \subseteq \{d \in \N \mid d \text{ divides } \size{A}\}\).
\end{cor}
\begin{lemma}	\label{lem:ReidemeisterNumbersExactSequence}
	Let \(G\) be a group, \(\phi \in \End(G)\) and \(N\) a \(\phi\)-invariant normal subgroup of \(G\) (\ie \(\phi(N) \leq N\)). Denote by \(\bar{\phi}\) the induced endomorphism on \(G / N\) and by \(\phi'\) the induced endomorphism on \(N\). Then \(R(\phi) \geq R(\bar{\phi})\).
	
	If, moreover, \(G\) is finite abelian, then \(R(\phi') \leq R(\phi)\).
\end{lemma}
\begin{proof}
	The first inequality is well-known, see \eg \cite[Lemma~1.1]{GoncalvesWong09}. For the second, if \(G\) is finite abelian, we know that \(R(\phi) = \size{\Fix(\phi)}\) and \(R(\phi') = \size{\Fix(\phi')}\). As \(\Fix(\phi') \leq \Fix(\phi)\), the inequality \(R(\phi') \leq R(\phi)\) follows.
\end{proof}

\begin{defin}
	Let \(A_{1}, \ldots, A_{n}\) be sets of natural numbers. We define
	\[
		\prod_{i = 1}^{n} A_{i} := \{a_{1} \ldots a_{n} \mid \forall i \in \{1, \ldots, n\}: a_{i} \in A_{i}\}.
	\]
	If \(A_{1} = \ldots = A_{n} =: A\), we also write \(A^{(n)}\).
\end{defin}

The following lemma can be found in \eg \cite[Corollary~2.6]{Senden21}.
\begin{lemma}	\label{lem:ReidemeisterNumbersOfAut(G1)xAut(Gn)}
	Let \(G_{1}, \ldots, G_{n}\) be groups and put \(G = \Times\limits_{i = 1}^{n} G_{i}\). Then
	\[
		\prod_{i = 1}^{n} \SpecR(G_{i}) \subseteq \SpecR(G).
	\]
	Equality holds if \(\Aut(G) = \Times\limits_{i = 1}^{n} \Aut(G_{i})\).
\end{lemma}

The following is well-known.
\begin{prop}
	Let \(G = G_{1} \times \ldots \times G_{n}\) be a direct product of finite groups such that \(\gcd(\size{G_{i}}, \size{G_{j}}) = 1\) for \(i \ne j\). Then
	\[
		\Aut(G) = \Times_{i = 1}^{n} \Aut(G_{i}).
	\]
\end{prop}

\begin{cor}	\label{cor:ReidemeisterSpectrumCoprimeGroups}
	Let \(G = G_{1} \times \ldots \times G_{n}\) be a direct product of finite groups such that \(\gcd(\size{G_{i}}, \size{G_{j}}) = 1\) for \(i \ne j\). Then
	\[
		\SpecR(G) = \prod_{i = 1}^{n} \SpecR(G_{i}).
	\]
\end{cor}
Since each finite abelian group \(A\) admits a unique decomposition of the form
\[
	A = \bigoplus_{p \in \P} A(p),
\]
where \(\P\) is the set of all primes and \(A(p)\) is the Sylow \(p\)-subgroup of \(A\), it is sufficient to determine the Reidemeister spectrum of finite abelian \(p\)-groups to completely determine the Reidemeister spectrum of finite abelian groups. For odd prime numbers, this is straightforward. For \(p = 2\), on the other hand, the situation is much more complicated, both the Reidemeister spectrum itself and the proof.

We write the cyclic group of order \(n\) as \(\ZnZ\) and write abelian groups additively.
\begin{lemma}	\label{lem:ReidemeisterNumberCyclicGroups}
	Let \(n \geq 2\) and let \(\phi \in \End(\ZnZ)\) be given by \(\phi(1) = k\). Then \(R(\phi) = \gcd(k - 1, n)\).
\end{lemma}
\begin{proof}
	Since \(R(\phi) = \size{\Fix(\phi)}\), we determine the fixed points of \(\phi\). We have that \(\phi(i) = i\) if and only if \(i \cdot (k - 1) \equiv 0 \bmod n\). Writing \(d = \gcd(k - 1, n)\), we see that \(\frac{k - 1}{d}\) is invertible modulo \(n\), hence \(i \cdot (k - 1) \equiv 0 \bmod n\) if and only if \(i \cdot d \equiv 0 \bmod n\). Thus, for \(i \cdot d \equiv 0 \bmod n\) to hold, \(i\) must be a multiple of \(\frac{n}{d}\). Since \(i\) has to lie between \(0\) and \(n - 1\) and there are \(d\) multiples of \(\frac{n}{d}\) lying between \(0\) and \(n - 1\), \(\phi\) has \(d\) fixed points.	
\end{proof}
We end with a general description of automorphisms of finite abelian \(p\)-groups, which was proven by C.\ Hillar and D.\ Rhea \cite{HillarRhea07}.
\begin{defin}
	Let \(n\) be a positive integer. We define \(E(n)\) to be
	\[
		E(n) = \{(e_{1}, \ldots, e_{n}) \in \Z^{n} \mid \forall i \in \{1, \ldots, n - 1\}: 1 \leq e_{i} \leq e_{i + 1}\}.
	\]
	
	Given a prime \(p\) and \(e \in E(n)\), we define the \emph{abelian \(p\)-group of type \(e\)} to be the group
	\[
		P_{p, e} = \bigoplus_{i = 1}^{n} \ZmodZ{p^{e_{i}}}.
	\]
\end{defin}
By the fundamental theorem of finite abelian groups, we know that for each finite abelian \(p\)-group \(P\) there exists an \(n \geq 1\) and \(e \in E(n)\) such that \(P \cong P_{p, e}\). We say that \(P\) is \emph{of type \(e\)}. We write \(\pi_{P}: \Z^{n} \to P\) for the natural projection. If \(P\) is clear from the context, we omit the subscript and simply write \(\pi\). Given \(\pi\), we write elements in \(\Z^{n}\) as column vectors \(\transpose{x}\).
\begin{theorem}[{\cite[Theorems 3.3 \& 3.6]{HillarRhea07}}]	\label{theo:AutomorphismGroupAbelianPGroup}
	Let \(P\) be a finite abelian \(p\)-group of type \(e\). Put \(A(P) = \{M \in \Z^{n \times n} \mid \forall j \leq i \in \{1, \ldots, n\}: p^{e_{i} - e_{j}} \mid M_{ij}\}\) and let \(\pi: \Z^{n} \to P\) be the natural projection. Define \(\Psi: A(P) \to \End(P): M \mapsto \Psi(M)\) where
	\[
		\Psi(M): P \to P: \pi(\transpose{x}) \mapsto \Psi(M)(\pi(\transpose{x})) := \pi(M \transpose{x}).
	\]
	Then \(A(P)\) is a ring under the usual matrix operations, \(\Psi\) is a well-defined ring morphism and \(\Aut(P)\) is precisely the image of \(\{M \in A(P) \mid M \bmod p \in \GL(n, \ZpZ)\}\) under \(\Psi\).
\end{theorem}
If \(\phi \in \Aut(P)\) is the image of \(M\) under \(\Psi\), we say that \(\phi\) is \emph{represented by \(M\)}.

For any abelian \(p\)-group \(P\) of type \(e \in E(n)\), the quotient group \(P / pP\) is an abelian group of exponent \(p\), hence it carries a \(\ZpZ\)-vector space structure. Note that the type of \(P / pP\) is then given by the all ones vector of length \(n\).
\begin{lemma}	\label{lem:InducedMatrixOnExponentpQuotient}
	Let \(P\) be a finite abelian \(p\)-group of type \(e\). Let \(\phi \in \Aut(P)\) be represented by \(M\). Let \(\transpose{z}_{i} \in \Z^{n}\) be the vector with a \(1\) on the \(i\)th place and zeroes elsewhere and let \(\rho: \Z^{n} \to P / pP\) be the projection. If we view \(P / pP\) as a vector space over \(\ZpZ\), the matrix representation of the induced automorphism on \(P / pP\) with respect to the basis \(\{\rho(\transpose{z}_{1}), \ldots, \rho(\transpose{z}_{n})\}\), is the matrix \(M \bmod p \in \GL(n, \ZpZ)\).
\end{lemma}
\begin{proof}
	Let \(\bar{\pi}: P \to P / pP\) be the projection and let \(\bar{\phi}\) be the induced automorphism on \(P / pP\). Then \(\bar{\pi} \circ \pi: \Z^{n} \to P / pP\) is the natural projection from \(\Z^{n}\) onto \(P / pP\), therefore, \(\rho = \bar{\pi} \circ \pi\). We also have that
	\[
		\bar{\phi}(\bar{\pi}(\pi(\transpose{x}))) = \bar{\pi}(\phi(\pi(\transpose{x}))) = \bar{\pi}(\pi(M\transpose{x})).
	\]
	Now, this implies that
	\[
		\bar{\phi}(\rho(\transpose{z}_{i})) = \rho(M\transpose{z}_{i}),
	\]
	showing that \(M \bmod p\) is the matrix representation of \(\bar{\phi}\).
\end{proof}


\section{Reidemeister spectrum of finite abelian \(p\)-groups with \(p\) odd prime}
For \(p\) an odd prime, the computation of the Reidemeister spectrum of a finite abelian \(p\)-group of type \(e\) is a straightforward application of \cref{lem:ReidemeisterNumberCyclicGroups,lem:ReidemeisterNumbersOfAut(G1)xAut(Gn)}.
\begin{lemma}	\label{lem:SpecROddCyclicpGroup}
	Let \(p\) be an odd prime and \(n \geq 1\) a natural number. Then
	\[
		\SpecR(\ZmodZ{p^{n}}) = \{p^{i} \mid i \in \{0, \ldots, n\}\}.
	\]
\end{lemma}
\begin{proof}
	The \(\subseteq\)-inclusion follows from \cref{cor:SpecRFiniteAbelianGroupsConsistsOfDivisors}. For the other inclusion, we use \cref{lem:ReidemeisterNumberCyclicGroups}. For \(i \in \{0, \ldots, n\}\), define \(\phi_{i}: \ZmodZ{p^{n}} \to \ZmodZ{p^{n}}: 1 \mapsto p^{i} + 1\). Since \(\gcd(p^{i} + 1, p) = 1\), as \(p\) is odd, \(\phi_{i}\) is an automorphism of \(\ZmodZ{p^{n}}\). By \cref{lem:ReidemeisterNumberCyclicGroups}, \(R(\phi_{i}) = p^{i}\).
\end{proof}
For \(e \in E(n)\) we put \(\Sigma(e) = \sum_{i = 1}^{n} e_{i}\).
\begin{prop}	\label{prop:SpecRFiniteAbelianOddPgroups}
	Let \(p\) be an odd prime and \(P\) a finite abelian \(p\)-group of type \(e \in E(n)\). Then
	\[
		\SpecR(P) = \{p^{i} \mid i \in \{0, \ldots, \Sigma(e)\}\}.
	\]
\end{prop}
\begin{proof}
	By \cref{cor:SpecRFiniteAbelianGroupsConsistsOfDivisors}, we only have to prove the \(\supseteq\)-inclusion, and is essentially \cref{lem:ReidemeisterNumbersOfAut(G1)xAut(Gn)}. Let \(m \in \{0, \ldots, \Sigma(e)\}\) and let \(j \in \{1, \ldots, n + 1\}\) be the (unique) index such that
	\[
		\sum_{l = 1}^{j - 1} e_{l} \leq m < \sum_{l = 1}^{j} e_{l},
	\]
	where we put \(e_{n + 1} := \infty\) for convenience. By \cref{lem:SpecROddCyclicpGroup}, there are automorphisms \(\phi_{i}\) of \(\ZmodZ{p^{e_{i}}}\) such that
	\[
		R(\phi_{i}) = \begin{cases}
			p^{e_{i}}	&\mbox{ if \(i \leq j - 1\)},	\\
			p^{m - \sum_{l = 1}^{j - 1}e_{l}}	&\mbox{ if \(i = j\)},	\\
			1	&\mbox{ if \(i > j\)}.
		\end{cases}
	\]
	Then \(\phi :=(\phi_{1}, \ldots, \phi_{n})\) is an automorphism of \(P\) and
	\[
		R(\phi) = \prod_{i = 1}^{n} R(\phi_{i}) = p^{\sum_{l = 1}^{j - 1}e_{l} + m - \sum_{l = 1}^{j - 1}e_{l}} = p^{m}.	\qedhere
	\]	
\end{proof}

\section{Fixed points on finite abelian \(p\)-groups}
In contrast to abelian \(p\)-groups for odd \(p\), the situation for abelian \(2\)-groups is much more involved. However, the behaviour of the Reidemeister numbers fits in a more general phenomenon concerning fixed points of automorphisms, valid for all prime numbers. Let \(p\) be a prime number, \(n \geq 1\) and \(e \in E(n)\). Let \(P := P_{p, e}\) be the finite abelian \(p\)-group of type \(e\). For \(i \in \Z\) coprime with \(p\), let \(\mu_{i}\) denote the automorphism of \(P\) given by \(\mu_{i}(x) = ix\). For \(\phi \in \Aut(P)\), we then define
\[
	\Pi(\phi) = \prod_{i = 1}^{p - 1} \size{\Fix(\mu_{i} \circ \phi)}.
\]
Finally, we put \(\SpecP(P) = \{\Pi(\psi) \mid \psi \in \Aut(P)\}\). The goal is to fully determine \(\SpecP(P)\). Note that \(\Pi(\phi)\) is always a power of \(p\), hence \(\SpecP(P) \subseteq \{p^{i} \mid i \in \N\}\). If \(p = 2\), then
\[
	\Pi(\phi) = \size{\Fix(\mu_{1} \circ \phi)} = \size{\Fix(\phi)} = R(\phi)
\]
by \cref{prop:ReidemeisterNumberEqualsNumberOfFixedPoints}, hence \(\SpecP(P) = \SpecR(P)\) in that case. This shows that \(\SpecP(P)\) is a generalisation of \(\SpecR(P)\) of some sort.

\subsection{Lower bound}
We start by determining and proving a lower bound for \(\SpecP(P)\). Throughout this section, let \(p\) be a prime number, \(n \geq 1\) and \(e \in E(n)\). Let \(P := P_{p, e}\) be the finite abelian \(p\)-group of type \(e\). To formulate the lower bound, we construct a decomposition of \(e\).

\begin{defin}	\label{def:abcDecomposition}
	Given \(e \in E(n)\), we construct the \emph{\(abc\)-decomposition of \(e\)} into three types of blocks in the following way.
\begin{description}
	\item[Step 1:] Each maximal constant subsequence of \(e_{1}, \ldots, e_{n}\) of length at least \(2\) forms one block, which we call an \emph{\(a\)-block}.
	
	\item[Step 2:] Among the remaining numbers, we look for successive numbers \(e_{i}\) and \(e_{i + 1}\) such that \(e_{i + 1} = e_{i} + 1\), starting from the left. Each such pair forms one block, which we call a \emph{\(b\)-block}.
	
	\item[Step 3:] By Step 1 and Step 2, the remaining \(e_{i}\)'s are all distinct and differ at least \(2\) from each other. Each of these numbers forms one block, which we call a \emph{\(c\)-block}.
\end{description}
	We define \(a(e), b(e)\) and \(c(e)\) to be the number of \(a\)-, \(b\)- and \(c\)-blocks, respectively, in this decomposition.
\end{defin}
For instance, consider \(e = (1, 1, 2, 3, 4, 4, 6, 7, 8, 10, 12, 13)\). We go through the steps one by one and mark the blocks in \(e\). There are two \(a\)-blocks, namely \((1, 1)\) and \((4, 4)\), hence we get
\[
	((1, 1), 2, 3, (4, 4), 6, 7, 8, 10, 12, 13).
\]
Next, there are three \(b\)-blocks, namely \((2, 3), (6, 7)\) and \((12, 13)\), so we get
\[
	((1, 1), (2, 3), (4, 4), (6, 7), 8, 10, (12, 13)).
\]
The remaining elements, \(8\) and \(10\), each form a single \(c\)-block, yielding
\[
	((1, 1), (2, 3), (4, 4), (6, 7), (8), (10), (12, 13)).
\]

\begin{remark}
This construction implies that, if a \(b\)-block of the form \((e_{i}, e_{i} + 1)\) succeeds a \(c\)-block \((e_{i - 1})\), then \(e_{i} \geq e_{i - 1} + 2\), since we form the \(b\)-blocks by starting from the left.
\end{remark}


We now use this decomposition to formulate the lower bound of \(\SpecP(P)\).

\begin{theorem}	\label{theo:LowerBoundNumberOfFixedPointsAbelianPGroup}
	Let \(\phi \in \Aut(P)\). Then
	\(
		\Pi(\phi) \geq p^{b(e) + c(e)}.
	\)
\end{theorem}
The remainder of this section is devoted to proving this theorem. To do so, we will construct a suitable characteristic subgroup of \(P\). This subgroup is of the following form.
\begin{defin}
	For non-negative integers \(d_{1}, \ldots, d_{n}\) with \(d_{i} \leq e_{i}\) for all \(i\), we define \(P(d_{1}, \ldots, d_{n})\) to be the subgroup
	\[
		\bigoplus_{i = 1}^{n} p^{d_{i}} \Z / p^{e_{i}} \Z
	\]
	of \(P\).
	
	Equivalently, if we let \(\pi: \Z^{n} \to P\) be the natural projection, then \(P(d_{1}, \ldots, d_{n}) = \pi(p^{d_{1}} \Z \oplus \ldots \oplus p^{d_{n}} \Z)\).
\end{defin}
\begin{theorem}	\label{theo:characteristicSubgroupsFiniteAbelianpGroup}
	Let \(d_{1}, \ldots, d_{n}\) be non-negative integers with \(d_{i} \leq e_{i}\) for all \(i\). Then \(Q := P(d_{1}, \ldots, d_{n})\) is characteristic in \(P\) if and only if the following two conditions hold:
	\begin{enumerate}[(i)]
		\item for all \(i \in \{1, \ldots, n - 1\}\) we have \(d_{i} \leq d_{i + 1}\).
		\item for all \(i \in \{1, \ldots, n - 1\}\) we have \(e_{i} - d_{i} \leq e_{i + 1} - d_{i + 1}\).
	\end{enumerate}
	Moreover, if \(Q\) is characteristic, \(d_{i} < e_{i}\) for all \(i \in \{1, \ldots, n\}\) and \(\phi \in \Aut(P)\) is represented by the matrix \(M\) as in \cref{theo:AutomorphismGroupAbelianPGroup}, then the induced automorphism on \(Q\) is represented by the matrix \(\inv{D} M D\), where \(D = Diag(p^{d_{1}}, \ldots, p^{d_{n}})\).
\end{theorem}
\begin{proof}
	For the first part, we use \cite[Theorem~2.2]{KerbyRode11}. There it is proven that the conditions on \(d_{1}, \ldots, d_{n}\) are equivalent with the subgroup \(P(e_{1} - d_{1}, \ldots, e_{n} - d_{n})\) being characteristic. However, if the \(n\)-tuple \(d := (d_{1}, \ldots, d_{n})\) satisfies the two conditions, then so does the \(n\)-tuple \(d' := (e_{1} - d_{1}, \ldots, e_{n} - d_{n})\), and vice versa. Indeed, the second condition for \(d\) implies the first one for \(d'\), and by symmetry, the first for \(d\) implies the second for \(d'\). Moreover, since \(0 \leq d_{i} \leq e_{i}\) for all \(i\), also \(0 \leq e_{i} - d_{i} \leq e_{i}\) for all \(i\). This proves the first part.
	
	Suppose now that \(Q\) is characteristic in \(P\) and that \(d_{i} < e_{i}\) for all \(i \in \{1, \ldots, n\}\). Fix \(\phi \in \Aut(P)\) and suppose that it is represented by \(M\). In order to use \cref{theo:AutomorphismGroupAbelianPGroup} to talk about the matrix representation of automorphisms of \(Q\), we have to write \(Q\) as a direct sum of cyclic groups of prime-power order. It is readily verified that
	\[
		\Phi : \bigoplus_{i = 1}^{n} \ZmodZ{p^{e_{i} - d_{i}}} \to \bigoplus_{i = 1}^{n} p^{d_{i}} \Z / p^{e_{i}} \Z: (x_{1}, \ldots, x_{n}) \mapsto (p^{d_{1}} x_{1}, \ldots, p^{d_{n}} x_{n})
	\]
	is an isomorphism, which implies that \(Q\) is an abelian \(p\)-group of type \((e_{1} - d_{1}, \ldots, e_{n} - d_{n})\). Let \(\tilde{Q}\) denote the group on the left-hand side.
	Write \(\pi_{P}: \Z^{n} \to P\) and \(\pi_{\tilde{Q}}: \Z^{n} \to \tilde{Q}\) for the natural projections onto \(P\) and \(\tilde{Q}\). Then \(\phi(\pi_{P}(\transpose{x})) = \pi_{P}(M \transpose{x})\) for all \(\transpose{x} \in \Z^{n}\). Let \(\phi_{Q}\) denote the induced automorphism on \(Q\) and put \(\psi := \inv{\Phi} \circ \phi_{Q} \circ \Phi\). Now, suppose that \(\transpose{x} \in \Z^{n}\) is such that \(\pi_{P}(\transpose{x}) \in Q\). Then we have \(x = (p^{d_{1}}y_{1}, \ldots, p^{d_{n}} y_{n})\) for some \(y_{1}, \ldots, y_{n} \in \Z\). Put \(y = (y_{1}, \ldots, y_{n})\). Then \(\transpose{x} = D\transpose{y}\) and therefore, \(\pi_{\tilde{Q}}(\transpose{y}) = \inv{\Phi}(\pi_{P}(\transpose{x}))\). Thus,
	\begin{align*}
		\psi(\pi_{\tilde{Q}}(\transpose{y}))	&= (\inv{\Phi} \circ \phi_{Q} \circ \Phi \circ \inv{\Phi})(\pi_{P}(\transpose{x}))	\\
									&= \inv{\Phi}(\phi_{Q}(\pi_{P}(\transpose{x})))	\\
									&= \inv{\Phi}(\phi(\pi_{P}(\transpose{x})))	\\
									&= \inv{\Phi}(\pi_{P}(M\transpose{x}))	\\
									&= \inv{\Phi}(\pi_{P}(MD \transpose{y})).
	\end{align*}
	Note that we can rewrite the equality \(\pi_{\tilde{Q}}(\transpose{y}) = \inv{\Phi}(\pi_{P}(\transpose{x}))\) as
	\[
		\pi_{\tilde{Q}}(\transpose{y}) = \inv{\Phi}(\pi_{P}(D\transpose{y})),
	\]
	which holds for arbitrary \(\transpose{y} \in \Z^{n}\). Since we know that \(\pi_{P}(MD \transpose{y}) \in Q\), we know that \(\inv{D}MD \transpose{y}\) is a well-defined element of \(\Z^{n}\). Thus, using the equalities above, we get
	\begin{align*}
		\psi(\pi_{\tilde{Q}}(\transpose{y}))	&= \inv{\Phi}(\pi_{P}(MD \transpose{y}))	\\
									&= \inv{\Phi}(\pi_{P}(D \inv{D}MD \transpose{y}))	\\
									&= \pi_{\tilde{Q}}(\inv{D}MD \transpose{y}).
	\end{align*}
	We conclude that the matrix representation of \(\psi\) is given by \(\inv{D} MD\), which finishes the proof.
\end{proof}
We now construct the aforementioned suitable characteristic subgroup by specifying the non-negative integers \(d_{i}\).

\begin{defin}
	Given \(e \in E(n)\) and its \(abc\)-decomposition as in \cref{def:abcDecomposition}, we define a new \(n\)-tuple \(d = (d_{1}, \ldots, d_{n})\) recursively. Put \(d_{1} = 0\). Given \(d_{i}\), we define
	\[
		d_{i + 1} = \begin{cases}
			d_{i}	&	\mbox{if \(e_{i}\) and \(e_{i + 1}\) lie in the same block or \(e_{i + 1}\) lies in an \(a\)-block}	\\
			d_{i} + 1	&	\mbox{if \(e_{i}\) and \(e_{i + 1}\) do not lie in the same block and \(e_{i + 1}\) lies in a \(b\)- or \(c\)-block.}
		\end{cases}
	\]
	We let \(d(e)\) denote this sequence.
\end{defin}
For example, given \(e = ((1, 1), (2, 3), (4, 4), (6, 7), (8), (10), (12, 13))\) as before with its \(abc\)-decomposition marked, we find that
\[
	d(e) = (0, 0, 1, 1, 1, 1, 2, 2, 3, 4, 5, 5).
\]
\begin{lemma}	\label{lem:Propertiesd(e)}
	Given \(e \in E(n)\), its associated \(n\)-tuple \(d(e)\) has the following properties:
	\begin{enumerate}[(i)]
		\item for all \(i, j \in \{1, \ldots, n\}\) with \(i < j\) we have \(d_{i} \leq d_{j}\) with strict inequality if \(e_{j}\) is the first element of a \(b\)- or \(c\)-block.
		\item for all \(i, j \in \{1, \ldots, n\}\) with \(i < j\) we have \(d_{j} - d_{i} \leq e_{j} - e_{i}\), with strict inequality if \(e_{i}\) is the first element of a \(b\)- or \(c\)-block.
		\item for all \(i \in \{1, \ldots, n\}\) we have \(d_{i} < e_{i}\).
	\end{enumerate}
\end{lemma}
\begin{proof}
	The sequence \(d(e)\) is non-decreasing by construction, proving the inequality in the first item. For the strictness part, note that it follows by construction if \(i = j - 1\), and the general case follows from the chain \(d_{i} \leq d_{j - 1} < d_{j}\).
	
	For the second item, we first prove it for \(j = i + 1\). By definition, we have
	\[
		d_{i + 1} - d_{i} = \begin{cases}
			0	&	\mbox{if \(e_{i}\) and \(e_{i + 1}\) lie in the same block or \(e_{i + 1}\) lies in an \(a\)-block}	\\
			1	&	\mbox{if \(e_{i}\) and \(e_{i + 1}\) do not lie in the same block and \(e_{i + 1}\) lies in a \(b\)- or \(c\)-block.}
		\end{cases}
	\]
	We now consider \(e_{i + 1} - e_{i}\). We distinguish several cases, based on the type of blocks in which \(e_{i + 1}\) and \(e_{i}\) lie.
	\begin{itemize}
		\item \(e_{i}\) and \(e_{i + 1}\) lie in the same \(a\)-block: then \(e_{i + 1} - e_{i} = 0\), by definition of an \(a\)-block.
		
		Since \(d_{i + 1} - d_{i} = 0\), we have \(d_{i + 1} - d_{i} \leq e_{i + 1} - e_{i}\). 
		\item \(e_{i}\) and \(e_{i + 1}\) lie in the same \(b\)-block: then \(e_{i + 1} - e_{i} = 1\), by definition of a \(b\)-block.
		
		Since \(d_{i + 1} - d_{i} = 0\), we have \(d_{i + 1} - d_{i} < e_{i + 1} - e_{i}\).
		\item \(e_{i}\) lies in an \(a\)- or \(b\)-block, \(e_{i + 1}\) does not lie in the same block: then \(e_{i + 1} - e_{i} \geq 1\), for otherwise \(e_{i + 1}\) and \(e_{i}\) would be part of an \(a\)-block. 
		
		Since \(d_{i + 1} - d_{i} \leq 1\), we have \(d_{i + 1} - d_{i} \leq e_{i + 1} - e_{i}\).
		\item \(e_{i}\) lies in a \(c\)-block, \(e_{i + 1}\) lies in an \(a\)-block: then \(e_{i + 1} - e_{i} \geq 1\) for the same reason as in the previous case.
		
		Since \(d_{i + 1} - d_{i} = 0\), we have \(d_{i + 1} - d_{i} < e_{i + 1} - e_{i}\).
		\item \(e_{i}\) lies in a \(c\)-block, \(e_{i + 1}\) lies in a \(b\)-block: then \(e_{i + 1} - e_{i} \geq 2\) by the remark following \cref{def:abcDecomposition}.
		
		Since \(d_{i + 1} - d_{i} = 1\), we have \(d_{i + 1} - d_{i} < e_{i + 1} - e_{i}\).
		\item \(e_{i}\) lies in a \(c\)-block, \(e_{i + 1}\) lies in a \(c\)-block: then \(e_{i + 1} - e_{i} \geq 2\), for otherwise \(e_{i + 1}\) and \(e_{i}\) would be part of an \(a\)-block or one or more \(b\)-blocks.
		
		Since \(d_{i + 1} - d_{i} = 1\), we have \(d_{i + 1} - d_{i} < e_{i + 1} - e_{i}\).
	\end{itemize}
	We see that in all cases the inequality \(d_{i + 1} - d_{i} \leq e_{i + 1} - e_{i}\) holds. Moreover, in the cases where \(e_{i}\) is the first element of a \(b\)- or \(c\)-block, we have proven that in fact the strict inequality holds. This finishes the proof for \(j = i + 1\).	
	
	We prove the general case by induction on \(j - i\), with base case \(j - i = 1\). Suppose it holds for all \(i < j\) with \(j - i < k\). Suppose that \(j - i = k\). Note that
	\[
		e_{j} - e_{i} - d_{j} + d_{i} = (e_{j} - e_{j - 1} - d_{j} + d_{j - 1}) + (e_{j - 1} - e_{i} + d_{i} - d_{j - 1}).
	\]
	Both terms on the right-hand side are non-negative by the induction hypothesis, hence the left-hand side is non-negative as well. Moreover, if \(e_{i}\) is the first element of a \(b\)- or \(c\)-block, then \(e_{j - 1} - e_{i} + d_{i} - d_{j - 1} > 0\), which implies that also \(e_{j} - e_{i} - d_{j} + d_{i} > 0\).
	
	Finally, for the third item, we again proceed by induction. For \(i = 1\), we have \(d_{1} = 0 < 1 \leq e_{1}\). So, suppose \(d_{i} < e_{i}\). Then by the second item, we know that \(d_{i + 1} - d_{i} \leq e_{i + 1} - e_{i}\). Adding the inequality \(d_{i} < e_{i}\) side by side yields \(d_{i + 1} < e_{i + 1}\).
\end{proof}

\begin{cor}
	The subgroup \(P(d_{1}, \ldots, d_{n})\) is a characteristic subgroup of \(P\).
\end{cor}
\begin{proof}
	By the previous lemma, \(d(e)\) satisfies all the conditions from \cref{theo:characteristicSubgroupsFiniteAbelianpGroup}.
\end{proof}
We will use the subgroup \(Q := P(d_{1}, \ldots, d_{n})\) to prove the lower bound on the number of fixed points.

\begin{lemma}	\label{lem:MatrixRepresentationModpOnQ}
	Let \(\phi \in \Aut(P)\) be represented by a matrix \(M \in \Z^{n \times n}\). Put \(D = Diag(p^{d_{1}}, \ldots, p^{d_{n}})\) and let \(i \ne j \in \{1, \ldots, n\}\). Then the following hold:
	\begin{enumerate}[(i)]
		\item \((\inv{D}MD)_{ij} \equiv 0 \bmod p\) if \(e_{j}\) is the first element of a \(b\)- or \(c\)-block.
		\item \((\inv{D}MD)_{jj} \not \equiv 0 \bmod p\) if \(e_{j}\) is the first element of a \(b\)- or \(c\)-block.
	\end{enumerate}
\end{lemma}
\begin{proof}
	For \(a \in \Z\), denote by \(\nu_{p}(a)\) the \(p\)-adic valuation of \(a\). First, remark that \((\inv{D}MD)_{ij} = \inv{D}_{ii} M_{ij}D_{jj}\), as \(D\) is diagonal. Next, by the properties of \(M\) and the definition of \(D\), we have that
	\[
		\nu_{p}((\inv{D} M D)_{ij}) = \nu_{p}(\inv{D}_{ii}M_{ij}D_{jj}) \geq \begin{cases}
			e_{i} - e_{j} + d_{j} - d_{i}	&	\mbox{if } i > j	\\
			 d_{j} - d_{i}				&	\mbox{if } i < j.
		\end{cases}
	\]
	Suppose that \(e_{j}\) is the first element of a \(b\)- or \(c\)-block. Then by \cref{lem:Propertiesd(e)}, each of the expressions above is at least \(1\). Therefore, \((\inv{D}MD)_{ij} \equiv 0 \bmod p\).
	
	For \((\inv{D} MD)_{jj}\), note that \(\inv{D}MD\) is the matrix representation of \(\phi_{Q}\), by \cref{theo:characteristicSubgroupsFiniteAbelianpGroup}. Moreover, it has to be invertible modulo \(p\) in order to define an automorphism on \(Q / pQ\). Since the \(j\)th column of \(\inv{D}MD\) is zero modulo \(p\) everywhere above and below the diagonal entry, the entry on the diagonal must be non-zero modulo \(p\).
\end{proof}
Finally, we prove \cref{theo:LowerBoundNumberOfFixedPointsAbelianPGroup}. For a matrix \(A \in \Z^{n \times n}\), we write \(\bar{A}\) for the matrix \(A \bmod p \in (\ZpZ)^{n \times n}\).
\begin{proof}[Proof of \cref{theo:LowerBoundNumberOfFixedPointsAbelianPGroup}]
	Let \(\phi \in \Aut(P)\) be represented by \(M \in \Z^{n \times n}\) and let \(d_{1}, \ldots, d_{n}, Q\) and \(D\) be as before. Since \(d_{i} < e_{i}\), the group \(Q\) has type \(\tilde{e} \in E(n)\). The matrix representation of \(\phi_{Q}\) is given by \(N := \inv{D} M D\), by \cref{theo:characteristicSubgroupsFiniteAbelianpGroup}. Let \(\bar{\phi}\) denote the induced automorphism on the exponent-\(p\) factor group \(Q / pQ\). 
	By \cref{lem:InducedMatrixOnExponentpQuotient}, the matrix representation of \(\bar{\phi}\) with respect to the basis \((1, 0, \ldots, 0), \ldots, (0, \ldots, 0, 1)\) is given by  \(\bar{N}\). 
By \cref{lem:MatrixRepresentationModpOnQ}, each column corresponding to a \(c\)-block and to a first element of a \(b\)-block in \(N\) is zero modulo \(p\), except for the element on the diagonal.
	
	Next, remark that for \(i \in \{1, \ldots, p - 1\}\), the automorphism \(\mu_{i}\) is represented by the matrix \(X_{i} := \Diag(i, \ldots, i)\). The automorphism \(\mu_{i} \circ \phi\) is then represented by \(X_{i} M\), and the one of the induced automorphism \(\bar{\mu}_{i} \circ \bar{\phi}\) on \(Q / pQ\) by \(\bar{X}_{i}\bar{N}\). Fix \(j \in \{1, \ldots, n\}\) such that \(e_{j}\) is the first element of a \(b\)- or \(c\)-block. Then \(N_{jj} \not \equiv 0 \bmod p\) by \cref{lem:MatrixRepresentationModpOnQ}, hence there is a unique \(i \in \{1, \ldots, p - 1\}\) such that \(iN_{jj} \equiv 1 \bmod p\). For that \(i\), we have that the \(j\)th column of \(\bar{X}_{i} \bar{N} - \bar{I}_{n}\) is zero.
		
	Now, let \(\J\) be the set of indices \(j\) such that \(e_{j}\) is the first element of a \(b\)- or \(c\)-block. For \(i \in \{1, \ldots, p - 1\}\), let \(\J_{i} = \{j \in \J \mid i N_{jj} \equiv 1 \bmod p\}\). Note that \(\J\) is the disjoint union of \(\J_{1}\) up to \(\J_{p - 1}\), and that \(\size{\J} = b(e) + c(e)\). Then by the arguments above, \(\bar{\mu}_{i} \circ \bar{\phi}\) has at least \(p^{\size{J_{i}}}\) fixed points. Indeed, for each \(j \in \J_{i}\), the \(j\)th column of \(\bar{X}_{i} \bar{N} - \bar{I}_{n}\) is zero, hence \(\ker(\bar{X}_{i} \bar{N} - \bar{I}_{n})\) has at least dimension \(\size{\J_{i}}\). By \cref{prop:ReidemeisterNumberEqualsNumberOfFixedPoints}, we know that \(R(\bar{\mu}_{i} \circ \bar{\phi}) = \size{\Fix(\bar{\mu}_{i} \circ \bar{\phi})}\) and \(R(\mu_{i} \circ \phi) = \size{\Fix(\mu_{i} \circ \phi)}\). By \cref{lem:ReidemeisterNumbersExactSequence}, we know that
	\[
		R(\bar{\mu}_{i} \circ \bar{\phi}) \leq R\left(\restr{(\mu_{i} \circ \phi)}{Q}\right) \leq R(\mu_{i} \circ \phi).
	\]
	Combining these inequalities, we conclude that
	\[
		\prod_{i = 1}^{p - 1} \size{\Fix(\mu_{i} \circ \phi)} \geq \prod_{i = 1}^{p - 1} p^{\size{\J_{i}}} = p^{\sum\limits_{i = 1}^{p - 1} \size{\J_{i}}} = p^{\size{\J}} = p^{b(e) + c(e)}.	\qedhere
	\]
\end{proof}
\subsection{Upper bound}
The next result provides an upper bound for \(\SpecP(P)\).
\begin{prop}	\label{prop:UpperBoundNumberOfFixedPointsAbelianPGroup}
	Let \(P\) be a finite abelian \(p\)-group of type \(e\). Let \(\phi \in \Aut(P)\). Then
	\(
		\Pi(\phi) \leq p^{\Sigma(e)}.
	\)
\end{prop}
\begin{proof}
	Fix \(\phi \in \Aut(P)\). We first prove that
	\[
		\Fix(\mu_{i} \circ \phi) \cap \grpgen{\Fix(\mu_{j} \circ \phi) \mid j \ne i}
	\]
	is trivial for all \(i \in \{1, \ldots, p - 1\}\). We proceed by induction, namely by proving that, for all \(k \in \{1, \ldots, p - 2\}\) and all \(\J \subseteq (\{1, \ldots, p - 1\} \setminus \{i\})\) with \(\size{\J} = k\), the intersection
	\[
		\Fix(\mu_{i} \circ \phi) \cap \grpgen{\Fix(\mu_{j} \circ \phi) \mid j \in \J}
	\]
	is trivial. We start with \(k = 1\), that is, with \(\J = \{j\}\) with \(j \ne i\). An element \(x\) in the intersection then satisfies \(x = i \phi(x) = j \phi(x)\), or equivalently, \((i - j) \phi(x) = 0\). As \(i \ne j\) and both lie in \(\{1, \ldots, p - 1\}\), we know that \(i - j\) is invertible modulo \(p\), hence \(\phi(x) = 0\). Since \(x = i \phi(x)\), we conclude that \(x = 0\). This proves the claim for \(k = 1\).
	
	Now, suppose that it holds for all \(\J\) of size \(k\) or less. Let \(\J\) be a set of size \(k + 1\) not containing \(i\) and let \(x\) be an element in the intersection \(\Fix(\mu_{i} \circ \phi) \cap \grpgen{\Fix(\mu_{j} \circ \phi) \mid j \in \J}\). Write \(x = \sum_{j \in \J} x_{j}\), with \(x_{j} \in \Fix(\mu_{j} \circ \phi)\). On the one hand, we have
	\[
		x = i \phi(x) =  i \sum_{j \in \J} \phi(x_{j}) = \sum_{j \in \J} i \phi(x_{j}),
	\]
	while on the other hand, we have
	\[
		x = \sum_{j \in \J} x_{j} = \sum_{j \in \J} j \phi(x_{j}).
	\]
	Therefore,
	\[
		0 = \sum_{j \in \J} (j - i) \phi(x_{j}).
	\]	
	Now, fix \(j_{0} \in \J\) and put \(\J' := \J \setminus \{j_{0}\}\). We can rewrite the equality above to
	\[
		(j_{0} - i) \phi(x_{j_{0}}) = \sum_{j \in \J'} -(j - i)\phi(x_{j}).
	\]
	Since \(\phi\) is an automorphism, we can apply \(\inv{\phi}\) to get
	\[
		(j_{0} - i) x_{j_{0}} = \sum_{j \in \J'} -(j - i)x_{j}
	\]
	The left-hand side lies in \(\Fix(\mu_{j_{0}} \circ \phi)\), the right-hand side is an element of \(\grpgen{\Fix(\mu_{j} \circ \phi) \mid j \in \J'}\). Applying the induction hypothesis to \(j_{0}\) and \(\J'\), we find that both sides are trivial, that is, \((j_{0} - i)x_{j_{0}} = 0 = \sum_{j \in \J'} -(j - i)x_{j}\). As \(i \ne j_{0}\), this implies \(x_{j_{0}} = 0\). Continuing in this fashion yields \(x_{j} = 0\) for all \(j \in \J\), finishing the induction. The original claim then follows from the case where \(\J = \{1, \ldots, p - 1\} \setminus \{i\}\).

From the above, it follows that
\[
	p^{\Sigma(e)} = \size{P}  \geq \size{\grpgen{\Fix(\mu_{i} \circ \phi) \mid i \in \{1, \ldots, p - 1\}}} = \prod_{i = 1}^{p - 1} \size{\Fix(\mu_{i} \circ \phi)} = \Pi(\phi),
\]
proving the upper bound.
\end{proof}
\subsection{Filling in the gaps}
We now completely determine \(\SpecP(P)\).
\begin{theorem}	\label{theo:BoundsProductNumberAreSharp}
	Let \(P\) be a finite abelian \(p\)-group of type \(e\). Then
	\[
		\SpecP(P) = \{p^{m} \mid m \in \{b(e) + c(e), \ldots, \Sigma(e)\}.
	\]
\end{theorem}
In order to prove this theorem, we first prove it for several special cases.
\begin{prop}	\label{prop:ProductNumberPrimePowerCyclicGroup}
	Let \(n \geq 1\) be a natural number. Then \(\SpecP(\ZmodZ{p^{n}}) = {\{p^{i} \mid i \in \{1, \ldots, n\}\}}\).
\end{prop}
\begin{proof}
	The \(\subseteq\)-inclusion follows from \cref{theo:LowerBoundNumberOfFixedPointsAbelianPGroup} and \cref{prop:UpperBoundNumberOfFixedPointsAbelianPGroup}. Conversely, let \(m \in \{1, \ldots, n\}\) be arbitrary. Define \(\phi_{m}: \ZmodZ{p^{n}} \to \ZmodZ{p^{n}}: 1 \mapsto p^{m} + 1\). Since \(m \geq 1\), we know that \(\gcd(p^{m} + 1, p) = 1\). Therefore, \(\phi_{m}\) defines an automorphism of \(\ZmodZ{p^{n}}\). Moreover, for \(i \in \{1, \ldots, p - 1\}\), we know by \cref{prop:ReidemeisterNumberEqualsNumberOfFixedPoints,lem:ReidemeisterNumberCyclicGroups} that
	\[
		\size{\Fix(\mu_{i} \circ \phi_{m})} = R(\mu_{i} \circ \phi_{m}) = \gcd(i(p^{m} + 1) - 1, p^{n}) = \begin{cases}
			p^{m}	& \mbox{if } i = 1	\\
			1		& \mbox{otherwise,}
		\end{cases}
	\]
	as \(i(p^{m} + 1) - 1 \equiv i - 1 \not \equiv 0 \bmod p\) when \(i \not \equiv 1 \bmod p\). Therefore, \(\Pi(\phi_{m}) = p^{m}\).
\end{proof}

\begin{lemma}	\label{lem:ProductNumberOfDirectProduct}
	Let \(P_{1}, \ldots, P_{n}\) be abelian \(p\)-groups and put \(P = P_{1} \oplus \ldots \oplus P_{n}\). For \(i \in \{1, \ldots, n\}\), let \(\phi_{i} \in \Aut(P_{i})\). Put \(\phi := (\phi_{1}, \ldots, \phi_{n}) \in \Aut(P)\). Then \(\Pi(\phi) = \prod_{i = 1}^{n} \Pi(\phi_{i})\).
	
	Consequently, \(\prod_{i = 1}^{n} \SpecP(P_{i}) \subseteq \SpecP(P)\).
\end{lemma}
\begin{proof}
	Let \(\mu_{i}\) be multiplication by \(i\) on \(P\) and let \(\mu_{i, j}\) denote its restriction to \(P_{j}\). We then have that
	\begin{align*}
		\Pi(\phi)	&=	\prod_{i = 1}^{p - 1} \size{\Fix(\mu_{i} \circ \phi)}	\\
				&=	\prod_{i = 1}^{p - 1} \size{\Fix\big((\mu_{i, 1} \circ \phi_{1}, \ldots, \mu_{i, n} \circ \phi_{n})\big)}	\\
				&=	\prod_{i = 1}^{p - 1} \prod_{j = 1}^{n} \size{\Fix(\mu_{i, j} \circ \phi_{j})}	\\
				&=	\prod_{j = 1}^{n} \prod_{i = 1}^{p - 1} \size{\Fix(\mu_{i, j} \circ \phi_{j})}	\\
				&=	\prod_{j = 1}^{n} \Pi(\phi_{j}).	\qedhere
	\end{align*}
\end{proof}
\begin{prop}	\label{prop:ProductNumberOfTypeb}
	Let \(n \geq 1\) be a natural number and put \(P = \ZmodZ{p^{n}} \oplus \ZmodZ{p^{n + 1}}\). Then \(\SpecP(P) = \{p^{i} \mid i \in \{1, \ldots, 2n + 1\}\}\).
\end{prop}
\begin{proof}
	Again, the \(\subseteq\)-inclusion follows from \cref{theo:LowerBoundNumberOfFixedPointsAbelianPGroup} and \cref{prop:UpperBoundNumberOfFixedPointsAbelianPGroup}. Conversely, let \(m \in \{1, \ldots, 2n + 1\}\). For \(m \geq 2\), we can find an automorphism \(\phi\) with \(\Pi(\phi) = p^{m}\) using  \cref{lem:ProductNumberOfDirectProduct} and \cref{prop:ProductNumberPrimePowerCyclicGroup}. Thus, suppose that \(m = 1\). Consider the matrix
	\[
		M = \begin{pmatrix} 1	&	1	\\
					p	&	1	\end{pmatrix}.
	\]
	By \cref{theo:AutomorphismGroupAbelianPGroup}, \(M\) defines an automorphism \(\phi\) on \(P\). First, we determine the fixed points of \(\phi\). If \(\phi(\pi(\transpose{(x, y)})) = \pi(\transpose{(x, y)})\), then
	\[
		\begin{cases}
			x + y \equiv x \bmod p^{n}	\\
			px + y \equiv y \bmod p^{n + 1}.
		\end{cases}
	\]
	This implies that \(y \equiv 0 \bmod p^{n}\) as well as \(x \equiv 0 \bmod p^{n}\). Therefore, the fixed points of \(\phi\) lie in the subgroup \(\grpgen{\pi(\transpose{(0, p^{n})})}\) and it is easily verified that \(\phi(\pi(\transpose{(0, p^{n})})) = \pi(\transpose{(0, p^{n})})\). Consequently, \(\size{\Fix(\phi)} = p\).
	
	Now, let \(i \in \{2, \ldots, p - 1\}\) and consider \(\mu_{i} \circ \phi\). If \((\mu_{i} \circ \phi)(\pi(\transpose{(x, y)})) = \pi(\transpose{(x, y)})\), then
	\[
		\begin{cases}
			ix + iy \equiv x \bmod p^{n}	\\
			ipx + iy \equiv y \bmod p^{n + 1}.
		\end{cases}
	\]
	The second congruence yields \((i - 1)y \equiv -ipx \bmod p^{n + 1}\). Since \(i \in \{2, \ldots, p - 1\}\), the number \(i - 1\) has an inverse  modulo \(p^{n + 1}\), say, \(j\). Substituting \(- jipx\) in the first congruence then yields
	\[
		x(i - i^{2}jp - 1) \equiv 0 \bmod p^{n}.
	\]
	Since \(i - i^{2} jp - 1 \equiv i - 1 \not \equiv 0 \bmod p\), it is invertible modulo \(p^{n}\). Therefore, \(x \equiv 0 \bmod p^{n}\). Combined with \((i - 1)y \equiv -ipx \bmod p^{n + 1}\) this yields \(y \equiv 0 \bmod p^{n + 1}\). Consequently, \(\size{\Fix(\mu_{i} \circ \phi)} = 1\). We conclude that \(\Pi(\phi) = p\).
\end{proof}

\begin{lemma}	\label{lem:FixedPointsOnExponentpGroup}
	Let \(n, k\) be integers with \(n \geq 2\), \(k \geq 1\). Put \(P = \bigoplus_{i = 1}^{n} \ZmodZ{p^{k}}\). Let \(\phi \in \Aut(P)\) and denote by \(\bar{\phi}\) the induced automorphism on \(P / pP\). If \(\bar{\phi}\) has no non-trivial fixed points, then neither does \(\phi\).
\end{lemma}
\begin{proof}
	We proceed by contraposition. Let \(\phi\) be represented by \(M\) and let \(\pi: \Z^{n} \to P\) be the natural projection. Suppose that \(M\transpose{x} \equiv \transpose{x} \bmod p^{k}\) for some \(\transpose{x} \in \Z^{n}\) with \(\pi(\transpose{x}) \ne 0\). Here, \(M\transpose{x} \equiv \transpose{x} \bmod p^{k}\) means that \((M\transpose{x})_{i} \equiv \transpose{x}_{i} \bmod p^{k}\) for each \(i \in \{1, \ldots, n\}\). Write \(x = p^{l} y\) with \(y \in \Z^{n}\) and \(l\) maximal. Then \(l < k\), otherwise \(\pi(\transpose{x}) = 0\). In particular, \(\pi(\transpose{y}) \ne 0\).
	
	Since \(M \transpose{x} \equiv \transpose{x} \bmod p^{k}\), we find \(p^{l}M\transpose{y} \equiv p^{l} \transpose{y} \bmod p^{k}\). Dividing by \(p^{l}\) yields \(M \transpose{y} \equiv \transpose{y} \bmod p^{k - l}\). As \(l < k\), we have that \(k - l \geq 1\). In particular, \(M \transpose{y} \equiv \transpose{y} \bmod p\). Thus, if \(\rho: P \to P / pP\) is the canonical projection, it follows that \(\rho(\pi(\transpose{y}))\) is a non-trivial fixed point of \(\bar{\phi}\).
\end{proof}

\begin{prop}	\label{prop:ProductNumberOfTypea}
	Let \(n, k\) be integers with \(n \geq 2\), \(k \geq 1\). Put \(P = \bigoplus_{i = 1}^{n} \ZmodZ{p^{k}}\). Then \(\SpecP(P) = \{p^{i} \mid i \in \{0, \ldots, nk\}\}\).
\end{prop}
\begin{proof}
	Yet again, the \(\subseteq\)-inclusion follows from \cref{theo:LowerBoundNumberOfFixedPointsAbelianPGroup} and \cref{prop:UpperBoundNumberOfFixedPointsAbelianPGroup}. Conversely, fix \(m \in \{0, \ldots, nk\}\). For \(m \geq n\), we can find an automorphism \(\phi\) with \(\Pi(\phi) = p^{m}\) using  \cref{lem:ProductNumberOfDirectProduct} and \cref{prop:ProductNumberPrimePowerCyclicGroup}. Thus, suppose that \(m \leq n - 1\).
	
	We start with \(m = 0\). Using a primitive element of the finite field of \(p^{n}\) elements, we can find a polynomial \(f_{n}\) of degree \(n\) that is irreducible over \(\ZpZ\). Its companion matrix \(C_{f_{n}}\) (seen as matrix over \(\Z\)) is invertible modulo \(p\). Consequently, it induces, by \cref{theo:AutomorphismGroupAbelianPGroup}, an automorphism \(\phi_{f_{n}}\) of \(P\). Since \(f_{n}\) has no roots in \(\ZpZ\) (recall that \(n \geq 2\)), the matrix \(C_{f_{n}}\) has no eigenvalues in \(\ZpZ\). Therefore, \(iC_{f_{n}}\) does not have eigenvalue \(1\) for \(i \in \{1, \ldots, p - 1\}\). Thus, \cref{lem:FixedPointsOnExponentpGroup} implies that \(\mu_{i} \circ \phi_{f_{n}}\) has no non-trivial fixed points for each \(i \in \{1, \ldots, p - 1\}\). Consequently, \(\Pi(\phi_{f_{n}}) = 1\).
	
	Now, we proceed for general \(n\). First, let \(n = 2\). We already know that \(\{1, p^{2}, p^{3}, \ldots, p^{2k}\}\subseteq \SpecP(P)\). Thus, we have to find an automorphism \(\psi\) such that \(\Pi(\psi) = p\). An argument similar to the one for \cref{prop:ProductNumberOfTypeb} shows that the automorphism \(\psi\) induced by the matrix
	\[
		M = \begin{pmatrix} 1 & 1 \\ p & 1 \end{pmatrix}
	\]
	does the job. Consequently, \(\SpecP(P) = \{p^{i} \mid i \in \{0, \ldots, 2k\}\}\) for \(n = 2\).
	
	So, let \(n \geq 3\) be arbitrary. If \(n\) is even, write
	\[
		P = \bigoplus_{i = 1}^{\frac{n}{2}} \left(\ZmodZ{p^{k}}\right)^{2}.
	\]
	Then the result for \(n = 2\) combined with \cref{lem:ProductNumberOfDirectProduct} implies that
	\begin{align*}
		\SpecP\left(\left(\ZmodZ{p^{k}}\right)^{2}\right)^{\left(\frac{n}{2}\right)}	&= \{p^{i} \mid i \in \{0, \ldots, 2k\}\}^{\left(\frac{n}{2}\right)}	\\
															&= \{p^{i} \mid i \in \{0, \ldots nk\}\}	\\
															&\subseteq \SpecP(P),
	\end{align*}
	proving the result for \(n\) even. Next, suppose that \(n\) is odd. We know that \({1 \in \SpecP(P)}\) by the case \(m = 0\) above. Write
	\[
		P = \ZmodZ{p^{k}} \oplus \bigoplus_{i = 1}^{\frac{n - 1}{2}} \left(\ZmodZ{p^{k}}\right)^{2}.
	\]
	Then the result for \(n = 2\) combined with \cref{lem:ProductNumberOfDirectProduct} and \cref{prop:ProductNumberPrimePowerCyclicGroup} yields
	\begin{align*}
		\SpecP(\ZmodZ{p^{k}}) \cdot \SpecP\left(\left(\ZmodZ{p^{k}}\right)^{2}\right)^{\left(\frac{n - 1}{2}\right)}	&= \{p^{i} \mid i \in \{1, \ldots, k\}\} \cdot \{p^{i} \mid i \in \{0, \ldots, 2k\}\}^{\left(\frac{n - 1}{2}\right)}	\\
								&= \{p^{i} \mid i \in \{1, \ldots, nk\}\}	\\
								& \subseteq \SpecP(P),
	\end{align*}
	which proves the result for \(n\) odd, thereby finishing the proof.

\end{proof}

Finally, we can completely determine \(\SpecP(P)\) for arbitrary finite abelian \(p\)-groups.
\begin{proof}[Proof of \cref{theo:BoundsProductNumberAreSharp}]
We factorise \(P\) using the \(abc\)-decomposition of \(e\), \ie we write
	\[
		P = \left(\bigoplus_{i = 1}^{a(e)} \left(\ZmodZ{p^{a_{i}}}\right)^{n_{i}}\right) \oplus \left(\bigoplus_{i = 1}^{b(e)} \left(\ZmodZ{p^{b_{i}}} \oplus \ZmodZ{p^{b_{i} + 1}}\right) \right) \oplus \left(\bigoplus_{i = 1}^{c(e)} \ZmodZ{p^{c_{i}}} \right)
	\]
	where \(n_{i} \geq 2\) for all \(i \in \{1, \ldots, a(e)\}\) and \(c_{i} \geq c_{i - 1} + 2\) for all \(i \in \{1, \ldots, c(e)\}\). By \cref{theo:LowerBoundNumberOfFixedPointsAbelianPGroup,prop:UpperBoundNumberOfFixedPointsAbelianPGroup}. we know that
	\[
		\SpecP(P) \subseteq \{p^{i} \mid i \in \{b(e) + c(e), \ldots, \Sigma(e)\}\}.
	\]
	Conversely, by \cref{lem:ProductNumberOfDirectProduct,prop:ProductNumberOfTypeb,prop:ProductNumberOfTypea,prop:ProductNumberPrimePowerCyclicGroup}, \(\SpecP(P)\) contains
	\begin{align*}
		&\prod_{i = 1}^{a(e)} \SpecP\left(\left(\ZmodZ{p^{a_{i}}}\right)^{n_{i}}\right) \cdot \prod_{i = 1}^{b(e)} \SpecP\left(\ZmodZ{p^{b_{i}}} \oplus \ZmodZ{p^{b_{i} + 1}}\right) \cdot \prod_{i = 1}^{c(e)} \SpecP\left(\ZmodZ{p^{c_{i}}} \right)	\\
		&= \prod_{i = 1}^{a(e)} \{p^{j} \mid j \in \{0, \ldots, a_{i}n_{i}\}\} \cdot \prod_{i = 1}^{b(e)} \{p^{j} \mid j \in \{1, \ldots, 2b_{i} + 1\}\} \cdot \prod_{i = 1}^{c(e)} \{p^{j} \mid j \in \{1, \ldots, c_{i}\}\}	\\
				&= \{p^{i} \mid i \in \{b(e) + c(e), \ldots, \Sigma(e)\}\},
	\end{align*}
proving the theorem.
\end{proof}
In particular, since \(\SpecP(P) = \SpecR(P)\) for finite abelian \(2\)-groups, we have the following.
\begin{cor}	\label{cor:ReidemeisterSpectrumFiniteAbelian2Group}
	Let \(P\) be a finite abelian \(2\)-group of type \(e\). Then
	\[
		\SpecR(P) = \{2^{i} \mid i \in \{b(e) + c(e), \ldots, \Sigma(e)\}\}.
	\]
\end{cor}

At last, by combining \cref{cor:ReidemeisterSpectrumCoprimeGroups,prop:SpecRFiniteAbelianOddPgroups,cor:ReidemeisterSpectrumFiniteAbelian2Group}, we can determine the Reidemeister spectrum of an arbitrary finite abelian group.
\begin{theorem}
	Let \(A\) be a finite abelian group. Suppose its Sylow \(2\)-subgroup is of type \(e\). Then
	\[
		\SpecR(A) = \{d \in \N \mid d \text{ divides } \size{A} \text{ and } \nu_{2}(d) \geq b(e) + c(e)\}.
	\]
\end{theorem}
\section*{Acknowledgements}
The author would like to thank Karel Dekimpe for his useful remarks and suggestions.

\printbibliography
\end{document}